\documentclass[letter,10pt]{article}
\usepackage{amssymb}
\usepackage{amsthm}
\usepackage{amsmath}
\usepackage{mathrsfs}
\usepackage{verbatim}
\usepackage{enumitem}

\usepackage{cite}
\usepackage{hyperref}

\usepackage{color}

\newtheorem{remark}{Remark}

\begin{document}
\title{Multi-marginal optimal transport on Riemannian manifolds\footnote{Y.-H.K. is supported in part by 
Natural Sciences and Engineering
Research Council of Canada (NSERC) Discovery Grants 371642-09 as well as Alfred P. Sloan research fellowship 2012--2014.  B.P. is pleased to acknowledge the support of a University of Alberta start-up grant and National Sciences and Engineering Research Council of Canada Discovery Grant number 412779-2012.}}\author{Young-Heon Kim\footnote{Department of Mathematics, University of British Columbia, Vancouver BC Canada V6T 1Z2 yhkim@math.ubc.ca} and Brendan Pass\footnote{Department of Mathematical and Statistical Sciences, 632 CAB, University of Alberta, Edmonton, Alberta, Canada, T6G 2G1 pass@ualberta.ca.}}
\maketitle
\begin{abstract}
We study a multi-marginal optimal transportation problem on a Riemannian manifold, with cost function given by the average distance squared from multiple points to their barycenter. 
Under a standard regularity condition on the first marginal, we prove that the optimal measure is unique and concentrated on the graph of a function over the first variable, thus inducing a Monge solution.  This result generalizes McCann's polar factorization theorem on manifolds from 2 to several marginals, in the same sense that a well known result of Gangbo and Swiech generalizes Brenier's polar factorization theorem on $\mathbb{R}^n$.
\end{abstract}

\section{Introduction}
In this paper, we study an optimal transport problem with several marginals. Given Borel probability measures $\mu_1,\mu_2,...,\mu_m$, each compactly supported on some smooth manifold $M$, and a cost function $c: M^m \rightarrow \mathbb{R}$, the multi-marginal optimal transport problem of Monge is to minimize

\begin{equation}\tag{\textbf{M}}\label{M}
\int_{M}c(x_1,F_2(x_1),...,F_m(x_1))d\mu_1(x_1)
\end{equation}
among $(m-1)$-tuples of mappings $(F_2,...,F_m)$, such that for each $i$, the map $F_i:M \mapsto M$ pushes the measure $\mu_1$ forward to $\mu_i$; that is, given any Borel $A \subseteq M$, $\mu_1(F_i^{-1}(A)) =\mu_i(A)$.  

The Kantorovich formulation of the multi-marginal optimal transport problem is the minimize
\begin{equation}\tag{\textbf{K}}\label{K}
\int_{M^m}c(x_1,x_2,...,x_m)d\gamma(x_1,x_2,...,x_m)
\end{equation}
among all probability measures $\gamma$ on the product space $M^m$ which project to the $\mu_i$; that is, such that $\gamma(M^{i-1}\times A \times M^{m-i}) = \mu_i(A)$ for all Borel $A \subseteq M$ and all $i=1,2,...,m$. Under reasonable conditions (e.g. continuity of the cost function and compactness of the supports of the measures), it is fairly straightforward to assert the existence of a minimizer $\gamma$ in \eqref{K}. Note that, for any $(F_2,...,F_m)$ satisfying the pushforward constraint in \eqref{M}, the pushforward measure $\gamma$ of $\mu_1$ by the map $(Id,F_2,...,F_M):M \rightarrow M^m$ satisfies the marginal constraint in \eqref{K} and so \eqref{K} is a relaxed version of \eqref{M}.

When $m=2$,  \eqref{K} and \eqref{M} correspond to the Monge and Kantorovich formulations, respectively, of the classical optimal transportation problem.  This problem has been studied extensively over the past 25 years, and has many applications and deep connections to various areas of mathematics.  Under a structural condition on the cost function, that is in particular satisfied by the Riemannian distance squared, and a fairly weak regularity condition on the first marginal, it is now well known that the optimal measure is concentrated on the graph $\{(x_1,F_2(x_1)):x_1 \in M\}$ of a map  $F_2:M \rightarrow M$.  The map $F_2$ minimizes \eqref{M}, and both $F_2$ and $\gamma$ are unique \cite{lev}\cite{G}\cite{GM}\cite{bren}\cite{m3}.

In recent years, the \emph{multi-marginal} case, $m\geq 3$ has attracted increasing attention, due to emerging applications in areas such as economics \cite{CE}\cite{CMN}, physics \cite{CFK}\cite{bdpgg}, cyclical monotonicity  \cite{GG, GhM, GhMa}, and systems of elliptic equations \cite{GhP1}.  In contrast to the two marginal case, the structure of solutions for general cost functions is not well understood.  For costs with certain special properties, however, a number of authors have proven that, like in the two marginal case, the optimal measure $\gamma$ is concentrated on the graph of a function over $x_1$ and is unique \cite{GS}\cite{C}\cite{P}\cite{P1}\cite{P9}.  On the other hand, it is known that for certain other cost functions, the optimal measure $\gamma$ in the $m \geq 3$ case may be concentrated on a set of Hausdorff dimension larger than $n:=$dim$(M)$ and be non unique \cite{CN}\cite{P}.  It remains an open question to determine precisely which costs permit Monge solution and uniqueness results and which do not.

The most important cost in the two marginal case is the quadratic cost, either $c(x_1,x_2) =|x_1-x_2|^2$ on $M \subseteq \mathbb{R}^n$, or $c(x_1,x_2) =d(x_1,x_2)^2$, where $d$ denotes the distance on a Riemannian manifold $M$.  Uniqueness and Monge solution results for these two cost functions constitute famous theorems of Brenier \cite{bren} and McCann \cite{m3}, respectively; these two results are at the heart of optimal transport theory and underlie many of it applications.

The best known result on multi-marginal problems concerns the cost function $\sum_{i \neq j} |x_i-x_j|^2$ on Euclidean space, $M \subseteq \mathbb{R}^n$.  For this cost, Gangbo and Swiech \cite{GS} proved uniqueness and existence of Monge solution, amounting to a generalization of Brenier's theorem from $2$ to $m \geq 3$ marginals.

The purpose of this paper is to generalize McCann's result on manifolds to the multi-marginal case, in the same spirit that Gangbo and Swiech's theorem generalizes the result of Brenier.  Alternatively, we can think of our main result as generalizing Gangbo and Swiech's result to manifolds, in the same sense that McCann's theorem generalizes Brenier's.  We will prove uniqueness and Monge solution results for the multi-marginal problem on a compact Riemannian manifold, with cost function

\begin{equation}\label{bccost}
c(x_1,x_2,...,x_m) = \inf_{y \in M} \sum_{i=1}^m \frac{d^2}{2}(x_i,y).
\end{equation}
 Let us note that, while the cost $\sum_{i \neq j} d(x_i,x_j)^2$ is a direct, algebraic generalization of the Gangbo-Swiech cost, the cost function \eqref{bccost} seems to be more intimately linked with the underlying geometry of the manifold $M$. 
This cost function 
measures the average distance squared between the points  $x_1$, $x_2$, ..., $x_m$ and their barycenter (also known as the Frechet or Karcher mean).   
Alternatively, we can view $(x_1, \cdots, x_m)$ as a point in the Riemannian product $M \times \cdots \times M$ equipped with the product distance $\sqrt{\sum_{i=1}^m \frac{d^2}{2} (x_i, y_i)}$ and the cost function in (\ref{bccost}) measures the distance from $(x_1, \cdots, x_n)$ to the diagonal set   $\{(y, \cdots, y)\}_{y\in M}$. 
Straightforward calculations in \cite{P9} show that the cost functions of Brenier (when $m=2$ on $\mathbb{R}^n$), McCann (when $m=2$ on a Riemannian manifold) and Gangbo-Swiech (for general $m$ on $\mathbb{R}^n$) are all equal to cost \eqref{bccost} in the appropriate settings, and so \eqref{bccost} for $m>2$ on a Riemannian manifold is indeed an extension of these costs.   Furthermore, there is an intimate link between optimal transport with cost \eqref{bccost} and the notion of the barycenter of the measures $\mu_1,\mu_2,...,\mu_2$ as considered by Agueh and Carlier \cite{ac} in the Euclidean case, following the work of Carlier and Ekeland \cite{CE} who considered more a general version of the cost function \eqref{bccost} in economic applications (see cost \eqref{mcost} below).

Our main theorem is the following.
\newtheorem{main}{Theorem}[section]
\begin{main}\label{main}
Assume $\mu_1$ is absolutely continuous with respect to local coordinates.  Then the solution $\gamma$ to \eqref{K} with cost function \eqref{bccost} is concentrated on the graph of a mapping $(F_2,F_3,...,F_m)$ over the first variable.  This  mapping is a solution to Monge's problem \eqref{M}, and the solutions to both \eqref{K} and \eqref{M} are unique.
\end{main}
This theorem is, to the best of our knowledge, the first of its kind for multi-marginal problems on manifolds with topology different from $\mathbb{R}^n$. We also note that one can extend Theorem \ref{main} to slightly more general costs, similar to \eqref{bccost}: see Theorem \ref{extension}. 

The Gangbo-Swiech optimal maps are all compositions of Brenier maps, $\nabla u_i^*\circ \nabla u_1$, where the $u_i^*$ and $u_1$ are convex functions; this structure was further clarified by Agueh and Carlier \cite{ac},  who showed that the first map, $\nabla u_1$, is the optimal map pushing $\mu_1$ forward to the Wasserstein \emph{barycenter} (see \eqref{matching}) of the measures $\mu_1,...\mu_m$, while the second map $\nabla u_i^*$ is the optimal mapping pushing the   barycenter forward to the measure $\mu_i$ \cite{ac}.  Our optimizers here take the same form, but using McCann maps rather than Brenier maps to account for the curved geometry: the optimizers are of the form 
 \begin{equation}\label{structure}
F_i (x_1) = \exp_{\exp_{x_1}(\nabla u_1(x_1))}(\nabla u_i^c(\exp_{x_1}(\nabla u_1(x_1))))
\end{equation}
for $\frac{d^2}{2}$-concave $u_1$, $u_i^c$; that is, they are compositions of McCann maps (see Proposition \ref{form}).  Therefore, our optimal map is built out of McCann maps, in the same sense that the Gangbo-Swiech maps are built from Brenier maps.

Let us mention that multi-marginal problems with cost functions of the general form 

\begin{equation}\label{mcost}
c(x_1,x_2,...,x_m) = \inf_{y}\sum_{i=1}^mc_i(x_i,y)
\end{equation}
arise naturally in matching problems in economics as considered by Carlier and Ekeland \cite{CE} and 
Chiapporri, McCann and Nesheim.
 \cite{CMN}. In this setting, solving the multi-marginal problem can be interpreted as finding an equilibrium assignment of $m$ different groups of workers to contracts $y$.

Motivated in part by this application, problems of costs of the form \eqref{mcost} were studied by one of us in a recent paper \cite{P9}, and conditions on the $c_i$'s were found that ensured existence and uniqueness of Monge solutions \cite{P9}.  However, that paper was restricted to Euclidean space, $M \subseteq \mathbb{R}^n$, and required twist and a type of non-degeneracy conditions on the cost functions $c_i$;  topological obstructions prevent these hypotheses from holding on, for example, a compact manifold.  It should be noted, however, that the results in \cite{P9} do apply to the quadratic cost \label{bccost} on bounded domains in a Hadamard manifold (a simply connected Riemannian manifold with nonpositive sectional curvature); in this case, the Cartan-Hadamard Theorem ensures that $M$ has the topology of $\mathbb{R}^n$.

Our method for the proof of Theorem \ref{main} is a bit different from the  approach in \cite{P9}, which has its origin in the papers of Carlier and Ekeland \cite{CE} and Agueh and Carlier \cite{ac}.  In  that work, it was crucial to have  the absolute continuity of the (generalized) barycenter measure with respect to local coordinates, since the optimal maps $F_i$'s were constructed from Brenier theory applied to  optimal maps from the barycenter measure, say $\nu$ (see \eqref{matching}), to the target measure $\mu_i$, whose existence requires absolute continuity of the source distribution $\nu$. In the Euclidean setting, the absolute continuity of $\nu$ has been shown by Agueh and Carlier \cite{ac}, but its Riemannian extension is still not known except the case of nonpositively curved manifold \cite{P9}.  In our method, we directly (without relying on absolute continuity of $\nu$) show that the solution $\gamma$ to \eqref{K} is concentrated on a graph of map. First, we will show that for $\mu_1$ almost every $x_1$, there is a unique point $y$ attaining the minimum in \eqref{bccost}, for $x_2,...,x_m$ such that $(x_1,x_2,...,x_m)$ is in the support of the optimal measure $\gamma$ in \eqref{K}.  We will then show, in turn, that for each $y$, there is at most $1$ corresponding point $(x_1,...,x_m)$ in the support of $\gamma$; together, these two facts will imply the main result. It is also remarkable that the structure of the maps $F_i$ as in \eqref{structure} has been obtained without relying on absolute continuity of $\nu$, while the maps $\exp_{ y} (\nabla u_i (y)$ in \eqref{structure} are Brenier maps from $\nu$ to $\mu_i$.

In Section 2, we recall the dual problem to \eqref{K} and some of its properties.  In the third section, we will prove several lemmas which are needed for the proof of Theorem \ref{main}, while the Theorem itself will be proved in Section 4.  In Section 5, we discuss the connection of \eqref{K} and \eqref{M} with the Wasserstein barycenter of the measures $\mu_1,...,\mu_m$, showing the structure \eqref{structure} of the maps $F_i$.  In the final section, we extend our results to other functions of the distance on a Riemannian manifold.

\section{Duality}
Here we recall known results on the Kantorovich dual problem, and make some remarks on the properties of its solution which will be relevant later on.  
 
The dual problem to \eqref{K} is to minimize,
\begin{equation}\tag{\textbf{D}}\label{D}
\sum_i^m\int_M u_i(x_i) d\mu_i(x_i)
\end{equation}
among all $m$-tuples of functions $(u_1,u_2,...,u_m)$ of functions, with $u_i \in L^1(\mu_i)$ and 
\begin{align}\label{eq:compatibility}
\sum_{i=1}^mu_i(x_i) \leq c(x_1,x_2,...,x_m)
\end{align}
 for $\otimes_{i=1}^m\mu_i$ almost everywhere point $(x_1, ... x_m)$.  We will say that an $m$-tuple $(u_1,u_2,...,u_m)$ is \textit{$c$-conjugate} if, for all $i=1,2,...,m$, we have
\begin{equation}\label{eq:conj}
u_i(x_i) = \inf_{x_j \in M, j \neq i}\big[c(x_1,x_2,...,x_m)-\sum_{j \neq i}u_j(x_j)\big].
\end{equation}
It is clear that any $c$-conjugate $(u_1,u_2,...u_m)$ satisfies  (\ref{eq:compatibility}),
 so is a viable competitor in \eqref{D}.

The following result is well known:

\newtheorem{dual}{Theorem}[section]\label{th:dual}
\begin{dual}
There exists a $c$-conjugate solution $(u_1,u_2,...,u_m)$ to \eqref{D}.  If $\gamma$ is any optimal measure in the Kantorovich problem, we have
\begin{equation*}
\sum_{i=1}^mu_i(x_i) = c(x_1,x_2,...,x_m)
\end{equation*}
$\gamma$-almost everywhere.

\end{dual}
It is also well known that if the cost $c$ is Lipschitz and $(u_1,u_2,...,u_m)$ is a $c$-conjugate $m$-tuple such that  if $u_1$ is not identically infinity, then $u_1$ must be Lipschitz \cite{m3}.  Our cost function \eqref{bccost} is Lipschitz, since it is defined as the infimum of a family of Lipschitz functions 
$(x_1, \cdots, x_m) \mapsto \sum_{i=1}^m \frac{d^2}{2}(x_i, y).$

\section{Properties of the cost function and consequences of optimality of $\gamma$}
 In this section, we establish several key lemmas for the proof of Theorem \ref{main}.
Throughout this section,  $c(x_1, ..., x_m)$  is the cost function  as in \eqref{bccost},  $\gamma$ denotes an optimal measure of the Kantorovich problem \eqref{K},  and $(u_1, ... , u_m)$ is a $c$-conjugate solution to \eqref{D}.  
\newtheorem{nocut}{Lemma}[section]
\begin{nocut}\label{nocut}
Fix $(x_1,...,x_m)$.  Then any $y$ which minimizes $y \mapsto \sum_{i=1}^m d^2(x_i,y)$  is not in the cut locus of $x_i$ for any $i$.
\end{nocut}
\begin{proof}
Choose a point $ y $ in the cut locus of $x_i$ for some $i$; we will show that $y$ cannot minimize $y \mapsto \sum_{i=1}^m d^2(x_i,y)$.  By Lemma 3.12 in \cite{c-ems}, we can find a constant $K$ such that, for all $u \in T_yM$, and  $j=1,2,...m$, we have

\begin{equation}
\frac{d^2(x_j, \exp_yu) + d^2(x_j, \exp_y(-u)) -2d^2(x_j,y)}{|u|^2} \leq K.
\end{equation}
On the other hand, by Proposition 2.5 in the same paper, we can find some non zero $u \in T_yM$ such that 

\begin{equation}
\frac{d^2(x_i, \exp_yu) + d^2(x_i, \exp_y(-u)) -2d^2(x_i,y)}{|u|^2} \leq -mK.
\end{equation}
Therefore, we have

\begin{eqnarray*}
\sum_{j=1}^md^2(x_j,y) &=& \sum_{j \neq i}^md^2(x_j,y) +d^2(x_i,y)\\
& \geq & \frac{-(m-1)K|u|^2}{2} + \frac{1}{2}\sum_{j \neq i}^m \big(d^2(x_j, \exp_yu) + d^2(x_j, \exp_y(-u))\big)\\
& +& \frac{mK|u|^2}{2} + \frac{1}{2}d^2(x_i, \exp_yu) + d^2(x_i, \exp_y(-u))\\
& >& \frac{1}{2}\sum_{j=1}^md^2(x_j, \exp_yu) + \frac{1}{2}\sum_{j=1}^md^2(x_j, \exp_y(-u)).
\end{eqnarray*}
Therefore, either 

\begin{equation*}
\sum_{j=1}^md^2(x_j, \exp_yu) <\sum_{j=1}^md^2(x_j,y),
\end{equation*}
or 
\begin{equation*}
\sum_{j=1}^m d^2(x_j, \exp_y-u) <\sum_{j=1}^m d^2(x_j,y);
\end{equation*}
in either case, $y$ cannot minimize  $\sum_{j=1}^m d^2(x_j,y)$.
\end{proof}

\newtheorem{semiccost}[nocut]{Lemma}
\begin{semiccost}
The cost function $c$ is everywhere superdifferentiable with respect to $x_1$.  That is, for all $(x_1,x_2,...,x_m) \in M^m$ there exist $p \in T_{x_1}M$ (the supergradient) such that, for small $v \in T_{x_1}M$, we have
\begin{equation*}
c(\exp_{x_1}v, x_2,...,x_m) \leq c(x_1, x_2,...,x_m) +g(p,v) + o(|v|),
\end{equation*}
where $g$ denotes the metric.
\end{semiccost}
\begin{proof}
Choose $y$ minimizing $y \mapsto \sum_i \frac{d^2}{2}^2(x_i,y)$.  By \cite{m3}, Proposition 6, the function $x_1 \mapsto d^2(x_1,y)$ is superdifferentiable.  Letting $p$ be it's supergradient, we have, for small $v$

\begin{equation*}
\frac{d^2}{2}^2(\exp_{x_1}v,y) \leq \frac{d^2}{2}(x_1,y) + g(p,v) + o(|v|).
\end{equation*}

Now, by definition, we have

\begin{eqnarray*}
c(\exp_{x_1}v, x_2,...,x_m) &\leq & \frac{d^2}{2}^2(exp_{x_1}v,y)+  \sum_{i=2}^m  \frac{d^2}{2}(x_i,y)\\
&\leq & \frac{d^2}{2}(x_1,y) + g(p,v) + o(|v|) +\sum_{i=2}^m  \frac{d^2}{2}^2(x_i,y)\\
&=&c(x_1, x_2,...,x_m)+ g(p,v) + o(|v|).
\end{eqnarray*}
Therefore, $c$ is superdifferentiable with respect to $x_1$, with supergradient $p$.
\end{proof}
\newtheorem{uniquey}[nocut]{Lemma}
\begin{uniquey}\label{uniquey}
At any point $(x_1,...,x_m)$ where $c$ is differentiable with respect to $x_1$, there is a unique minimizing $y$ in \eqref{bccost},  and moreover, 
\begin{align}\label{eq:y}
 y=\exp_{x_1}(\nabla _{x_1} c(x_1,...,x_m)).
\end{align}
\end{uniquey}

\begin{proof}
For any minimizing $y$ in \eqref{bccost}, $d^2(x_1,y)$ is differentiable as $y \notin cut(x_1)$ by Lemma \ref{nocut}. We then have 
\begin{equation*}\label{firstorder}
\nabla _{x_1}c(x_1,...,x_m)  = \nabla _{x_1}\big(\frac{1}{2}d^2(x_1,y) \big).
\end{equation*}
 This equation implies that $y$ must equal $\exp_{x_1}(\nabla _{x_1} c(x_1,...,x_m))$; uniqueness follows immediately.

\end{proof}

\begin{remark}\label{rm: nocut}
This lemma may be of independent interest.  The minimizing $y$ in \eqref{bccost} is known as the Frechet or Karcher mean of the points $x_1,x_2,...,x_m$, and can be interpreted as a generalization of the notion of average to a curved space.  Uniqueness of Frechet means is linked the curvature of the manifold $M$.  For example, Frechet means are easily seen to be unique on a Hadamard manifold, but may be non-unique for certain $x_i$ on spaces with positive sectoinal curvature; for example, every point on the equator is a Frechet mean of the north and south poles on the sphere.  The preceding lemma yields a sufficient condition for the Frechet mean of the points $x_1,x_2,...,x_m$ to be unique; namely, that the function $c$ be differentiable with respect to one of the variables.   We will see later that this in fact implies that, if $\mu_1$ does not charge small sets,  the Frechet mean of the points $(x_1,x_2,...,x_m)$ is unique $\gamma$ almost everywhere, where $\gamma$ is the optimal measure in \eqref{K}.
\end{remark}

\newtheorem{aediff}[nocut]{Lemma}
\begin{aediff}\label{aediff}
Let $(x_1,x_2,...,x_m) \in spt(\gamma)$, and suppose the potential $u_1$ is differentiable at $x_1$.  Then the cost $c$ is differentiable with respect to $x_1$ at $(x_1,x_2,...,x_m)$ and moreover,
\begin{align}\label{eq:grad u}
\nabla_{x_1} c (x_1, ... , x_m) =  \nabla u_1 (x_1) 
\end{align}
\end{aediff}
\begin{proof}
The proof is based on a now classical argument of McCann \cite{m3}.   By  the properties  \eqref{eq:compatibility} and \eqref{eq:conj}, Theorem~\ref{th:dual}, and the fact that $(x_1, ... , x_m) \in spt(\gamma)$, the gradient $q=\nabla u_1(x_1)$ serves as a subgradient for $c$ with respect to $x_1$. By lemma \ref{aediff}, $c$ has a supergradient, $p$, as well.  Therefore, we conclude that $c$ must be differentiable with respect to $x_1$, with gradient $p=q =\nabla u_1(x_1)$.
\end{proof}

Whenever the optimal $y$ in \eqref{bccost} is unique, we will denote it by $\overline{y}(x_1,...,x_m)$.
 The following proposition implies injectivity of this correspondence from the support of $\gamma$. 
\newtheorem{uniquex_i}[nocut]{Lemma}\begin{uniquex_i}\label{uniquex_i}
Suppose  $x=(x_1, \cdots, x_m)$ and $\bar x= (\bar x_1, \cdots ,\bar x_m)$ are both in spt($\gamma$) and there exists $y \in M$ which minimizes both $z \mapsto \sum_{i=1}^m d^2(x_i,z)$  and $z \mapsto \sum_{i=1}^m d^2(\bar x_i,z)$.  Then $x =\bar x$.
\end{uniquex_i}
\begin{proof}
Fix some $i$.  We will show  $x_i = \bar x_i$.

Define $x'= (x_1, \cdots, \bar x_i, \cdots  , x_k)$ and $\bar x'= (\bar x_1, \cdots, x_i, \cdots, \bar x_k)$.  Recalling that $c(x) = \inf_{z \in M} \sum \frac{d^2}{2}(x_i, z)$, we have
\begin{align}\label{mono}
c(x') + c(\bar x')  & \leq  \Big[\sum_{j \ne i} \frac{d^2}{2} (x_j, y)\Big] + \frac{d^2}{2} (\bar x_i, y)  + 
\Big[\sum_{j \ne i} \frac{d^2}{2} (\bar x_j, y) \Big] + \frac{d^2}{2} (x_i, y)\\\nonumber
&= c(x) +  c(\bar x),
\end{align}
where the last equality follows from rearranging the terms and using the assumption 
\begin{equation*}
y \in {\rm argmin}_{z \in M} \sum_{i=1}^m \frac{d^2}{2}(x_i,z) \bigcap {\rm argmin}_{z \in M}  \sum_{i=1}^m \frac{d^2}{2}(\bar x_i,z).
\end{equation*}

Also, the optimality of $\gamma$ implies a certain $c$-monotonicity property \cite{P}; as $x, \bar x \in spt \gamma$, this yields,
\begin{equation*}
c(x') + c(\bar x') \ge c(x) + c (\bar x).
\end{equation*}

Moreover, from \eqref{mono}, this implies that $c(x') =[\sum_{j \ne i} \frac{d^2}{2} (x_j, y)] + \frac{d^2}{2} (\bar x_i, y)$; that is,

\begin{equation*}
y \in {\rm argmin}_{z \in M} \Big [(\sum_{j\ne i}  \frac{d^2}{2} (x_j, z))  + \frac{d^2}{2} (\bar x_i, z)\Big]
\end{equation*}
 
The Riemannian distance is smooth away from the cut locus,  so, by Lemma \ref{nocut}, each $z \mapsto \frac{d^2}{2}(x_i,z)$  and each $z \mapsto \frac{d^2}{2}(\bar x_i,z)$ is differentiable at $z =y$.  Therefore, as $y \in {\rm argmin}_z [\sum_{j} \frac{d^2}{2} (x_j, z) ]$
 and $y \in {\rm argmin}_z \big[(\sum_{j\ne i} \frac{d^2}{2} (x_j, z) ) + \frac{d^2}{2} (\bar x_i, z)\big]$,
\begin{equation*}
\sum_j \nabla_y \frac{d^2}{2} (x_j , y) = 0  = \sum_{j\ne i} \nabla_y \frac{d^2}{2}  (x_j, y) + \nabla_y \frac{d^2}{2}  (\bar x_i, y).
\end{equation*}
This implies
\begin{equation*}
\nabla_y \frac{d^2}{2} (x_i, y) = \nabla_y \frac{d^2}{2} (\bar x_i, y)
\end{equation*}

It is well known that for $x \notin {\rm cut}(y)$, $x = \exp_y(\frac{1}{2}\nabla_y d^2 (x, y))$ and so the above implies

\begin{equation*}
x_i = \exp_y(\frac{1}{2}\nabla_y d^2 (x_i, y)) = \exp_y(\frac{1}{2}\nabla_y d^2 (\bar x_i, y))= \bar x_i.
\end{equation*}
as desired.
\end{proof}

\section{Main result}
We are now ready to prove Theorem \ref{main}.
\begin{proof}
We first consider the assertion that optimal measures must be concentrated on a graph.

For any optimal measure $\gamma$, we need to show that, for $\mu_1$ almost every $x_1$, there is a unique $(x_2,x_3,...,x_m)$ such that $(x_1,x_2,x_3,...,x_m)) \in spt(\gamma)$.
As the optimal measure $\gamma$ projects to $\mu_1$, for $\mu_1$ a.e. $x_1$, we must have \emph{at least one} $(x_2,x_3,...,x_m)$ such that $(x_1,x_2,x_3,...,x_m) \in spt(\gamma)$.  Therefore, it remain only to prove uniqueness of the $(x_2,x_3,...,x_m)$.

As mentioned in Section 2, it is well known that the Kantorovich potential $u_1$ is Lipschitz, and hence differentiable $\mu_1$ almost everywhere \cite{m3}.  Now, at every point where the potential is differentiable, the cost $c$ is differentiable with respect to $x_1$, by Lemma \ref{aediff}.   Thus from \eqref{eq:y} and \eqref{eq:grad u}, we see 
 that there is a unique $y$ minimizing \eqref{bccost}, for  all   
 $(x_1,x_2,x_3,...,x_m) \in spt(\gamma)$; namely  
 \begin{equation*}
y
= \exp_{x_1}(\nabla u(x_1)).
\end{equation*}
It follows from Lemma \ref{uniquex_i}, that for this $y$, there is {\em at most one} $$(x_1,x_2,...,x_m) \in spt(\gamma)$$ with $y \in {\rm argmin} \sum_{i}d^2(x_i,y)$,  showing the uniqueness of $(x_2, ..., x_m)$.  Denote this point $$G(y) := (G_1(y),G_2(y),....G_m(y)).$$  Then, denoting $F_i(x_i) = G_i(\exp_{x_1}(\nabla u(x_1)))$, $F_i$ is a well defined $\mu_1$-almost everywhere map with the desired properties.

It remains to prove uniqueness of the optimal measure $\gamma$ and the optimal maps $F_i$; our argument for this is standard in optimal transport theory.  Suppose there are two distinct optimizers, $\gamma$ and $\bar\gamma$; by the argument above, they are concentrated on graphs, $(F_2,F_3,...,F_m)$ and $(\bar F_2,\bar F_3,...,\bar F_m)$, respectively.  Now, by linearity of the Kantorovich functional, the interpolant $\frac{1}{2}\gamma + \frac{1}{2}\bar\gamma$ is also so optimal and must also be concentrated on a graph.  This immediately implies that $(F_2,F_3,...,F_m)=(\bar F_2,\bar F_3,...,\bar F_m)$ almost everywhere, completing the proof.
\end{proof}

\section{Barycenters in Wasserstein space on Riemannian manifolds}

In this section, we discuss the connection of our result with the Wasserstein barycenter of $\mu_1,...\mu_2$.  The barycenter (with equal weights) is defined as the Borel probability measure $\nu$ on $M$ which minimizes

\begin{equation}\label{matching}
\nu \mapsto \sum_{i=1}^m W^2_2(\mu_i, \nu),
\end{equation}
where $W^2_2(\mu_i, \nu)$ denotes the square of the quadratic Wasserstein distance from $\mu_i$ to $\nu$.    The relevance of the barycenter to the multi-marginal optimal mapping problem \eqref{M} was observed and investigated by Carlier and Ekeland \cite{CE} and by Agueh and Carlier \cite{ac}.
Existence of the barycenter follows easily from a continuity compactness argument (see \cite{CE}, Theorem 3).  When $\mu_1$ is absolutely continuous, uniqueness of the barycenter was shown in \cite{ac} for the Euclidean case $M\subseteq \mathbb{R}^n$, and can be easily generalized to the Riemannian case either from Theorem 3.4.1 in \cite{P5}, or alternatively, by adapting Proposition 4 in \cite{CE}, using McCann's theorem \cite{m3} in place of the twist, or generalized Spence-Mirrlees, condition. 

Letting $\gamma$ denote the optimal measure in the multi-marginal problem \eqref{K}, and $\bar y(x_1,x_2,...,x_m)$ the minimizer of $y \mapsto \sum_{i=1}^m d^2(x_i,y)$ (which, by Lemma \ref{uniquey}, is unique for $\gamma$ almost all $(x_1,x_2,...,x_m)$), a result of Carlier and Ekeland  \cite{CE} implies that

\begin{equation*}
\nu:=\overline{y}_{\#} \gamma
\end{equation*}
is the unique barycenter\footnote{Note that the minimizer of $y \mapsto \sum_{i=1}^m d^2(x_i,y)$ is not necessarily unique for \emph{all} $(x_1,x_2,...,x_m)$ here, as was assumed by Carlier and Ekeland in \cite{CE}.  However, it is clear by examining their proof that it is sufficient to have uniqueness for $\gamma$ almost all $(x_1,x_2,...,x_m)$.}.   They also showed that, for each $i$, the measure $\gamma_i$ defined on $M \times M$ by

\begin{equation*}
\gamma_i = (\pi_i, \bar y)_{\#}\gamma
\end{equation*}
is optimal for the two marginal Monge-Katorovich problem

\begin{equation}\label{twomarg}
\inf_{\gamma_i} \int_{M \times M}d^2(x_i,y) d\gamma_i
\end{equation}
with marginals $\mu_i$ and $\nu$.  Here $\pi_i(x_1,x_2,...,x_m) = x_i$ is the canonical projection.

The following proposition further highlighs the relationship between the barycenter and the optimal maps in Theorem \ref{main}.
\newtheorem{form}{Proposition}[section]\label{form}
\begin{form}\label{form}
Use the assumption and notation as in Theorem \ref{main}. Then, 
the mappings $F_i(x_1)$ in Theorem \ref{main} are of the form $G_i \circ G_1^{-1}$, where $G_1^{-1}$ is the optimal map (for the quadratic cost $d^2(x_i,y)$) pushing $\mu_1$ forward to the barycenter $\nu$, and $G_i$ is the optimal map pushing $\nu$ forward to the measure $\mu_i$.
\end{form}

\newtheorem{rem}[form]{Remark}
\begin{rem}
This generalizes the result of Gangbo and Swiech \cite{GS}, who showed that when $M \subseteq \mathbb{R}^n$, the optimal maps for \eqref{M} take the form $\nabla u_i^*\circ \nabla u_1$, for convex functions $u_i$ and their conjugates $u_i^*$.  Agueh and Carlier later realized that the maps $Du_i^*$ were in fact the Brenier (optimal) maps pushing the barycenter forward to $\mu_i$, and $Du_1$ the Brenier map pushing $\mu_1$ forward to the barycenter \cite{ac}.
\end{rem}

\begin{proof}
Letting $(u_1,u_2,...,u_m)$ be a c-conjugate solution to the dual problem \eqref{D}, we have

\begin{eqnarray}
u_i(x_i)& =& \inf_{x_j, j\neq i} \left[ c(x_1,x_2,...,x_m) -\sum_{i \neq j}u_j(x_j)\right]\\
&=& \inf_{x_j, j\neq i} \left[\left( \inf_y \sum_{k=1}^m \frac{d^2}{2}(x_k,y) \right)-\sum_{j \neq i}u_j(x_j)\right]\\
&=& \inf_y \left[ \frac{d^2}{2}(x_i,y) + \sum_{j \neq i}\inf_{x_j}\left[\frac{d^2}{2}(x_j,y)-u_j(x_j)\right]\right].\label{cconcave}
\end{eqnarray}
This implies that $u_i$ is $\frac{d^2}{2}$-concave.  Setting $v_i(y) = -\sum_{j \neq i}\inf_{x_j}\frac{d^2}{2}(x_j,y)-u_j(x_j)$, we have
\begin{equation}\label{twomargpotential}
u_i(x_i) + v_i(y) \leq \frac{d^2}{2}(x_i,y)
\end{equation}
and we have equality when 
\begin{equation*}
x \in S:=\{x=(x_1,...,x_m): c(x_1,x_2,...,x_m) -\sum_{j=1}^m u_j(x_j) \},
\end{equation*}
that is, $(u_i,v_i)$ solve the Kantorovich dual problem to \eqref{twomarg}.

Now, from the proof of Theorem \ref{main}, we have $F_i(x_1) = G_i(\exp_{x_1}(\nabla u_1(x_1))$ (for a.e $x_1$ thus, for $\mu_1$-a.e $x_1$ by our assumption that $\mu_1$ is absolutely continuous with respect to local coordinates). Also,  from \eqref{cconcave}, $u_1$ is $\frac{d^2}{2}$-concave, therefore by McCann's theorem the map $\exp_{x_1}(\nabla u_1(x_1))$ is optimal.  As for  $\mu_1$-a.e. $x_1$ we have equality in \eqref{twomargpotential} only when $y = \bar y(x_1,F_2(x_1),...,F_m(x_1))$, it is clear that $\exp_{x_1}(\nabla u_1(x_1))_{\#}\mu_1 = \bar y_{\#}\gamma =\nu$.

It remains to show that the maps $G_i$ are optimal maps pushing $\nu$ to $\mu_i$.   Let $u_i^c(y) :=\inf_{x_i \in M} \frac{d^2}{2}(x_i,y) -u_i(x_i)$ be the $\frac{d^2}{2}$ concave conjugate of $u_i$.  Note that we have $v_i(y) \leq u_i^c(y)$, and we have equality whenever there is some $x_i$ achieving equality in \eqref{twomargpotential}, which happens whenever $y=\bar y(x)$ for some $x \in S$. 

Standard arguments now imply that $G_i(y) = \exp_{y}(\nabla u_i^c(y))$ wherever $u_i^c$ is differentiable; we must show that this is $\nu$-almost everywhere.  Although the semi-concave function $u_i^c$ is differentiable almost everywhere with respect to local coordinates, the conclusion is not completely obvious, as we do not know that the measure $\nu$ does not charge small sets.
 (This latter fact is known for the Euclidean case, or equivalently for the Gangbo-\'Swi\c{e}ch cost  $\sum_{i \neq j} |x_i-x_j|^2$ , as was shown by Agueh and Carlier \cite{ac}.) This difficulty can easily be overcome thanks to  Lemma \ref{uniquex_i}. Details follow:

 Fix $y \in \bar y(S)$. (Notice that for $\nu$-almost everywhere points $y$,  we have $y\in \bar y(S)$.)
  Note that the c-concave function $u_i^c$ is superdifferentiable everywhere   (since the relevant domains, i.e. support of measures, etc,  are compact), and  that at any point $x_i$ where 
\begin{equation*}
u_i(x_i) +u_i^c(y) = \frac{d^2}{2}(x_i,y),
\end{equation*}
the vector $\nabla_{y} \frac{d^2}{2}(x_i,y)$ is in the superdifferential $\partial u_i^c(y)$.  Also,  as is well-known for concave and for $c$-concave functions, we see 
 that 
at any $x_i$, if   $\nabla_{y} \frac{d^2}{2}(x_i,y)$ is an \emph{extremal} point of the convex set  $\partial u_i^c(y)$, we must have 

\begin{equation*}
u_i(x_i) +u_i^c(y) = \frac{d^2}{2}(x_i,y).
\end{equation*}
Furthermore, as for $y \in \bar y(S)$ we have $v(y) = u_i^c(y)$, we have equality in \eqref{twomargpotential} for each such $x_i$.

In sum, for each $y \in \bar y(S)$
 we have shown that for each $x_i$, if  $\nabla_{y} \frac{d^2}{2}(x_i,y)$ is extremal in $ \partial u_i^c(y)$, then we have equality in \eqref{twomargpotential}. But by Lemma \ref{uniquex_i} there is exactly one point $x_i \in M$ for which equality holds in \eqref{twomargpotential}.
It follows that, at each point $y \in \bar y (S)$,  $\partial u_i^c(y)$ has only one extremal point, hence $u_i^c$ is differentiable on $\bar y (S)$.  Since $ y \in \bar y (S)$ for $\nu$-a.e. points $y$,   it now follows that 
\begin{equation*}
G_i(y) = \exp_y (\nabla u_i^c(y))
\end{equation*}
for $\nu$-a.e. $y$ and so that 
\begin{equation*}
F_i(x_1) = \exp_{\exp_{x_1}\nabla u_1(x_1)} (\nabla u_i^c(\exp_{x_1}\nabla u_1(x_1)))
\end{equation*}
for $\mu_1$-a.e. $x_1$ as desired. 
\end{proof}

\section{Extension to other functions of the distance}
In this section we consider the extension of our main theorem to non-quadratic cost functions.  Precisely, given $C^2$, strictly increasing, strictly convex functions $f_i:[0,\infty) \rightarrow \mathbb{R}$, for $i=1,2,...,m$,  we consider cost functions of the form

\begin{equation}\label{extcost}
c(x_1,x_2,...,x_m) =\inf_{y \in M} \sum_{i=1}^m f_i(d(x_i,y)).
\end{equation}
Note that the previous part of the paper deals with the case $f_i(t)=t^2$.

\newtheorem{extension}{Theorem}[section]
\begin{extension}\label{extension}
Assume $\mu_1$ is absolutely continuous with respect to local coordinates.  Then the solution $\gamma$ to \eqref{K} with cost function \eqref{extcost} is concentrated on the graph of a function $(F_2,F_3,...,F_m)$ over the first variable.  This function is a solution to Monge's problem \eqref{M}, and the solutions to both \eqref{K} and \eqref{M} are unique.
\end{extension}

\begin{proof}
The proof is very similar to the proof of Theorem \ref{main}, and we only sketch it here, explaining how to deal with the main differences.  The main difference is to verify that a minimizer $y \in {\rm argmin}\sum_{i=1}^m f_i(d(x_i,y))$ is not in the cutlocus of any $x_i$ (an analogue of Lemma \ref{nocut}).  This follows from a Taylor expansion. Details are given below:

Note that, by Taylor's theorem, for some $p \in [d(x_i, \exp_{y}(u)), d(x_i, y)]$

\begin{eqnarray}
& &f_i(d(x_i, \exp_{y}(u))) -f_i(d(x_i,y))\\& =& f_i'(d(x_i,y))\big( d(x_i, \exp_{y}(u)) -d(x_i,y) \big)\nonumber\\ 
& & + \ f_i''(p)\frac{( d(x_i, \exp_{y}(u)) -d(x_i,y) \big)^2}{2}\nonumber\\
&=& \frac{f_i'(d(x_i,y))}{d(x_i, \exp_{y}(u)) +d(x_i,y)}\big( d^2(x_i, \exp_{y}(u)) -d^2(x_i,y) \big)\nonumber\\
 && + \ f_i''(p)\frac{( d(x_i, \exp_{y}(u)) -d(x_i,y) \big)^2}{2}.\label{posu}
\end{eqnarray}

Similarly, for some $q \in [d(x_i, \exp_{y}(-u)), d(x_i, y)]$
\begin{eqnarray}
& & f_i(d(x_i, \exp_{y}(-u))) -f_i(d(x_i,y))\\
&=& \frac{f_i'(d(x_i,y))}{d(x_i, \exp_{y}(-u)) +d(x_i,y)}\big( d^2(x_i, \exp_{y}(-u)) -d^2(x_i,y) \big)\nonumber \\
 & &  \ + \ f_i''(q)\frac{( d(x_i, \exp_{y}(-u)) -d(x_i,y) \big)^2}{2}.\label{negu}
\end{eqnarray}
  By monotonicity of $f_i$, we have $f_i'(d(x_i,y)) >0$ and by convexity we have $f_i''(p), f_i''(q) \geq 0$.   Also, we have  (e.g. see \cite{m3}), 
\begin{equation*}
( d(x_i, \exp_{y}(-u)) -d(x_i,y) \big)^2 \leq C |u|^2 +o(|u|).
\end{equation*} 
 Note that near the cutlocus, the factor $d(x_i, \exp_{y}(u)) +d(x_i,y)$ is bounded below by $I$, the injectivity radius of the compact manifold $M$ (away from the cutlocus, inequality \eqref{uppbound} below follows easily from smoothness).  Using Lemma 3.12 in \cite{c-ems}, and combining \eqref{posu} and \eqref{negu}, we obtain

\begin{equation}\label{uppbound}
 f_i(d(x_i, \exp_{y}(u)))  + f_i(d(x_i, \exp_{y}(-u))) -2f_i(d(x_i,y)) \leq K |u|^2 + o(|u|)^2
\end{equation}
for some $K>0$.
Now, if $y \in {\rm cut}(x_j)$, we can, by Proposition 2.5 in \cite{c-ems}, for any $A>0$, choose a $u$ such that  
\begin{equation} \label{d2bound}
d^2(x_j, \exp_yu) + d^2(x_j, \exp_y(-u)) -2d^2(x_i,y)\leq -A|u|^2.
\end{equation}

Now, the factor $d(x_i, \exp_{y}(u)) +d(x_i,y)$ is bounded above by $2R$, where $R$ is the radius of the compact manifold.  Therefore, for any $B>0$, combining \eqref{posu} and \eqref{negu} and \eqref{d2bound}, we can find arbitrarily small $u$ such that
\begin{equation*}
 f_i(d(x_j, \exp_{y}(u)))  + f_i(d(x_j, \exp_{y}(-u))) -2f_i(d(x_j,y)) \leq -B|u|^2 + o(|u|)^2.
\end{equation*}
Now, choosing $B$ large enough, and arguing as in Lemma \ref{nocut}, we obtain either 

\begin{equation*}
 \sum_{i=1}^m f_i(d(x_i, \exp_{y}(u))) < \sum_{i=1}^mf_i(d(x_i,y))
\end{equation*}
or

\begin{equation*}
 \sum_{i=1}^m f_i(d(x_i, \exp_{y}(-u)))< \sum_{i=1}^mf_i(d(x_i,y)).
\end{equation*}

In either case, $y \notin {\rm argmin} \sum_{i=1}^m f_i(d(x_i,y))$.  The remainder of the proof, including analogues of Lemmas \ref{uniquey} and \ref{uniquex_i}, is very similar to the arguments developed earlier in the paper and is omitted.
\end{proof}
\bibliographystyle{plain}
\bibliography{biblio}

\end{document}